\documentclass[11pt]{amsart}
\usepackage[english]{babel}
\usepackage[utf8]{inputenc}
\usepackage{graphicx}
\usepackage{xcolor}
\usepackage{longtable}
\usepackage{float}
\usepackage[all]{xy}
\usepackage{amssymb,amscd}
\usepackage{tikz}
\usepackage{csquotes}
\usepackage{imakeidx}
\usepackage{appendix}
\usepackage{tikz-cd}
\usepackage{enumitem}
\usepackage[colorlinks=true,linkcolor=blue,citecolor=blue,urlcolor=magenta,pagebackref=true]{hyperref}
\usepackage{calc}
\usepackage{mathtools}
\usepackage{comment}
\usepackage{cleveref}

\numberwithin{equation}{section}

\newtheorem{theorem}{Theorem}[section]
\newtheorem{proposition}[theorem]{Proposition}
\newtheorem{lemma}[theorem]{Lemma}
\newtheorem{corollary}[theorem]{Corollary}

\theoremstyle{definition}

\theoremstyle{remark}
\newtheorem{remark}[theorem]{Remark}

\newcommand{\Gm}{\mathbb G_{\mathrm m}}
\newcommand{\GO}{\operatorname{GO}}
\newcommand{\GOp}{\operatorname{GO}^{+}}
\newcommand{\PGOp}{\operatorname{PGO}^{+}}
\newcommand{\SO}{\operatorname{SO}}
\newcommand{\Spin}{\operatorname{Spin}}
\newcommand{\Cor}{\operatorname{cor}}
\newcommand{\End}{\operatorname{End}}

\title{The norm principle for extended Clifford groups}
\author{Federico Scavia}
\address{CNRS, Institut Galil\'ee, Universit\'e Sorbonne Paris Nord, 99 avenue Jean-Baptiste Cl\'ement, 93430 Villetaneuse, France}
\email{scavia@math.univ-paris13.fr}
\date{September 24, 2026}
\subjclass[2020]{Primary 20G15; Secondary 11E04, 11E57, 14L15}
\keywords{Norm principle, quadratic form, spin group, extended Clifford group, $R$-equivalence, maximal torus}

\begin{document}
	
	\begin{abstract}
We prove the norm principle for the extended Clifford group $\Omega(q)$ of every nondegenerate quadratic form $q$ of even dimension at least $4$ over an arbitrary field $F$ of characteristic different from $2$, for all finite separable field extensions $L/F$. The main new ingredient is a factorization theorem for proper similitudes: for a quadratic extension $L/F$ and $f\in \mathrm{GO}^+(q)(L)$, there exist $u\in \mathrm{SO}(q)(L)$, a maximal $F$-torus $S\subset \mathrm{GO}^+(q)$, and $s\in S(L)$ such that $f=us$.
	\end{abstract}
	
	\maketitle

	\section{Introduction}
	
	Let $F$ be a field, let $G$ be a reductive $F$-group, let $T$ be an
	$F$-torus, and let $\varphi\colon G\to T$ be an $F$-homomorphism.
	Given a finite separable extension $L/F$, the \emph{norm principle} for
	$\varphi$ over $L/F$ is the inclusion
	\[
	N_{T,L/F}(\varphi(G(L)))\subset \varphi(G(F)),
	\]
	where $N_{T,L/F}\colon T(L)\to T(F)$ denotes the norm map. We say that
	the norm principle holds for $G$ over $L/F$ if it holds for every
	$F$-homomorphism from $G$ to an $F$-torus.
	
	Barquero--Merkurjev \cite[Theorem~1.1]{BM00} proved the norm principle
	for reductive groups whose Dynkin diagram has no component of type
	$D_n$ for $n\geq4$, $E_6$, or $E_7$. Thus, among the classical groups, the remaining case is type $D_n$. For groups arising from even-dimensional quadratic forms, the key case is the extended Clifford group $\Omega(q)$ of a nondegenerate quadratic form $q$; see
	\cite[\S2.1]{BCM2019}.
	
	The norm principle for $\Omega(q)$ was previously known in a number of
	important cases, including the case when $q$ is isotropic and the case when $F$ has small virtual cohomological dimension; see below for precise references. The anisotropic case, however, remained open. The purpose of this paper is to settle it.
	
	\begin{theorem}\label{thm:main}
		Let $F$ be a field of characteristic different from $2$, and let $q$ be
		a nondegenerate quadratic form of even dimension at least $4$ over $F$.
		For every finite separable extension $L/F$, the norm principle holds for
		the extended Clifford group $\Omega(q)$ over $L/F$.
	\end{theorem}
	
	Theorem~\ref{thm:main} has consequences for the norm principle for more general reductive groups and for Serre's injectivity question; see Corollaries~\ref{cor:reductive-norm-principle} and \ref{cor:serre-injectivity}.
	
	We place Theorem~\ref{thm:main} in the context of the classical
	norm principles for quadratic forms. Scharlau's norm principle
	\cite[Chapter~VII, \S4.3]{Lam2005} asserts that the norm of a similarity
	factor of $q_L$ is again a similarity factor of $q$; equivalently,
	\[N_{L/F}(\mu(\GO(q)(L)))\subset \mu(\GO(q)(F)),\]
	where $\mu\colon\GO(q)\to\Gm$ is the multiplier map. Similarly,
	Knebusch's norm principle asserts that norms of spinor norms of $q_L$
	are spinor norms of $q$. In terms of the even Clifford group $\Gamma^+(q)$, this is
	the norm principle for the spinor-norm homomorphism
	$\Gamma^+(q)\to\Gm$. The norm principle for $\Omega(q)$ expresses a
	subtle compatibility between these two classical norm principles; see
	\cite[Theorem~2.5]{BCM2019}.
	
	Before the present work, Theorem~\ref{thm:main} was known in some special cases:
	\begin{itemize}
		\item[--] when $F$ is a number field, it follows from work of Gille \cite{Gille1997};
		\item[--] when $\dim(q)=4$ or $6$, by Barquero--Merkurjev \cite{BM00};
		\item[--] when $q$ is isotropic, by Bhaskhar--Chernousov--Merkurjev \cite{BCM2019};
		\item[--] over fields of virtual cohomological dimension at most $2$, again  by Bhaskhar--Chernousov--Merkurjev \cite{BCM2019}.
	\end{itemize}
	Moreover, Bhaskhar--Chernousov--Merkurjev reduced the norm principle over a complete discretely valued field to the corresponding problem for quadratic forms over finite extensions of its residue field.
	
	The key new ingredient in the proof of Theorem~\ref{thm:main} is the following factorization result for proper similitudes over quadratic extensions.
	
	\begin{theorem}\label{thm:toral-replacement}
		Let $F$ be an infinite field of characteristic different from $2$, let $L/F$ be a quadratic extension, and let $q$ be a nondegenerate quadratic form of dimension $2n$ over $F$. For every $f\in\GOp(q)(L)$, there exist $u \in\SO(q)(L)$, a maximal $F$-torus $S\subset\GOp(q)$ and an element $s\in S(L)$ such that $f=us$.
	\end{theorem}
	
	We indicate how Theorem~\ref{thm:toral-replacement} implies Theorem~\ref{thm:main}. By the reduction theorem of Bhaskhar--Chernousov--Merkurjev (Theorem~\ref{thm:quadratic-reduction}), we may assume that $L/F$ is a quadratic extension. Another result of Bhaskhar--Chernousov--Merkurjev (Lemma~\ref{lem:bcm-equivalence}) translates the obstruction to the norm principle in terms of the central isogeny
	\[
	\Spin(q)\longrightarrow\PGOp(q).
	\]
	After lifting an element of $\PGOp(q)(L)$ to $f\in\GOp(q)(L)$, Theorem~\ref{thm:toral-replacement} gives a factorization $f=us$ with $u\in\SO(q)(L)$ and $s$ contained in a maximal $F$-torus. Since $\SO(q)$ is $F$-rational, the factor $u$ is $R$-trivial, and hence it is handled using Gille's Theorem~\ref{thm:gille}; the term $s$ is handled by the usual compatibility of corestriction with the norm map for tori. These two ingredients together imply the conclusion; see Corollary~\ref{cor:r-torus-corestriction}.
	
	We now outline the proof of Theorem~\ref{thm:toral-replacement}. Let $f\in\GOp(q)(L)$ and set $\lambda\coloneqq\mu(f)$. It is enough to find a maximal $F$-torus $S\subset\GOp(q)$ and an element $s\in S(L)$ with $\mu(s)=\lambda$, since then $u\coloneqq fs^{-1}$ belongs to $\SO(q)(L)$. If $\lambda=c^2$ in $L$, we may simply take $s=cI$. From now on, assume that $\lambda$ is not a square.
	
	We construct $S$ from an \'etale subalgebra of $\End_F(V)$ stable under the adjoint involution. Since $\lambda$ is a multiplier, Lemma~\ref{lem:self-adjoint-root} gives a self-adjoint similitude $A\in\GO(q)(L)$ of multiplier $\lambda$, that is, such that $A^2=A^*A=\lambda I$. Let $\sigma$ be the nontrivial element of $\mathrm{Gal}(L/F)$ and define
	\[
	P\coloneqq A\sigma(A),\qquad\nu\coloneqq\lambda\sigma(\lambda)\in F^\times.
	\]
	Then $\sigma(P)=P^*=\nu P^{-1}$ and, after replacing $A$ by a suitable $\SO(q)(L)$-conjugate,
	Proposition~\ref{prop:regular-product} allows us to assume that $P$ is
	regular semisimple and that its characteristic polynomial $\chi_P(t)$ is prime to
	$t^2-\nu$. Proposition~\ref{prop:spectral-torus}, applied to $P$, now gives an
	\'etale $F$-algebra $E\coloneqq L[P]^\sigma\subset\End_F(V)$
	with involution $\tau=(\mathrm{ad}_q)|_E$, together with a maximal
	$F$-torus $S\subset\GOp(q)$: for every $k$-algebra $R$, 
	\[S(R)=\{z\in(E\otimes_F R)^\times:z\tau(z)\in R^\times\}
	\subset\GOp(q).\]
	 It also gives an isometry
	$r\in\operatorname O(q)(F)$ whose conjugation action on $E$, and hence
	on $L[P]$, is the adjoint involution.
	
	On the other hand, conjugation by $A$ also induces the adjoint
	involution on $L[P]$. Hence $s\coloneqq Ar$ centralizes $L[P]$. Since $P$ is regular
	semisimple, $L[P]$ is self-centralizing, so $s\in L[P]^\times$.
	Moreover,
	\[s^*s=r^*A^*Ar=\lambda I.\]
	Therefore $s\in S(L)$ and $\mu(s)=\lambda$. Thus $fs^{-1}\in\SO(q)(L)$, as desired.
	
	The paper is organized as follows. In Section~\ref{sec:prelim}, we recall basic facts about quadratic forms, orthogonal and spin groups, and the extended Clifford group. In Section~\ref{sec:norm-prelim}, we review the cohomological formulation of the norm principle, the reduction to quadratic extensions, and the relevant transfer result for $R$-equivalence. In Section~\ref{sec:etale-tori}, we describe the maximal tori used in the proof of Theorem~\ref{thm:toral-replacement}. Section~\ref{sec:self-adjoint} establishes the results we need on self-adjoint similitudes and generic products. Finally, in Section~\ref{sec:toral-replacement}, we prove Theorem~\ref{thm:toral-replacement} and then deduce Theorem~\ref{thm:main}.
	
	\subsection*{Notation}
	
	Let $F$ be a field. An $F$-variety is a separated $F$-scheme of finite type. For an $F$-torus $T$ and a finite separable field extension $L/F$, we let $N_{T,L/F}\colon T(L)\to T(F)$ be the norm map.
	
	\section*{Acknowledgments}
	
	I thank Alexander Merkurjev for introducing me to this problem and for generously sharing his insights with me. 
	
	\section{Preliminaries on quadratic forms and orthogonal groups}\label{sec:prelim}
	
	\subsection{Quadratic spaces and similitudes}
	
	Let $F$ be a field of characteristic not equal to $2$, let $V$ be a finite-dimensional vector space over $F$, and let $q\colon V\to F$ be a nondegenerate quadratic form. Its associated symmetric bilinear form is
	\[b\colon V\times V\to F,\qquad b(x,y)=\frac{q(x+y)-q(x)-q(y)}{2}.\]
	For $a\in\End_F(V)$, its adjoint $a^*$ is
	characterized by
	\[
	b(ax,y)=b(x,a^*y)\qquad \forall x,y\in V.
	\]
	Recall that $a$ is self-adjoint if $a^*=a$, and skew-adjoint if $a^*=-a$. We write $\operatorname{ad}_q$ for the adjoint involution on $\End_F(V)$ associated with $q$, so that $\operatorname{ad}_q(a)=a^*$ for all $a\in \End_F(V)$.
	
	A similitude of $q$ is an element $g\in\operatorname{GL}(V)$
	such that $q(gx)=\mu(g)q(x)$ for some $\mu(g)\in F^\times$ and all
	$x\in V$, or equivalently $g^*g=\mu(g)I$. The scalar $\mu(g)$ is called the multiplier of $g$. The similitude group is	denoted by $\GO(q)$, and its subgroup of isometries by
	$\operatorname O(q)=\ker(\mu)$.
	
	Suppose that $\dim V=2n$. Taking determinants gives $\det(g)^2=\mu(g)^{2n}$, so $\det(g)=\pm\mu(g)^n$. The group of proper similitudes and the special orthogonal group are
	\[
	\GOp(q)\coloneqq\{g\in\GO(q):\det(g)=\mu(g)^n\},\qquad\SO(q)=\operatorname O(q)\cap\GOp(q).
	\]
	The scalar transformations form a central subgroup
	$\Gm\subset\GOp(q)$. We write $\PGOp(q)\coloneqq\GOp(q)/\Gm$.
	Hilbert's Theorem~90 implies that the map $\GOp(q)(F)\to\PGOp(q)(F)$ is surjective.
	
	\subsection{Spin groups}\label{subsec:spin}
	
	Assume that $V$ has even dimension at least $4$. We have a central short exact sequence 
	\begin{equation}\label{eq:spin-central-isogeny}
		1\longrightarrow Z\longrightarrow\Spin(q)\longrightarrow\PGOp(q)\longrightarrow 1,\qquad Z=Z(\Spin(q)).
	\end{equation}
	Here $Z$ is a finite commutative \'etale group of $F$-order $4$; see
	\cite[\S23]{KMRT1998}. For every extension $K/F$, let
	\[
	\delta_K\colon\PGOp(q)(K)\longrightarrow H^1(K,Z),
	\qquad
	\alpha_K\colon H^1(K,Z)\longrightarrow H^1(K,\Spin(q))
	\]
	be the connecting map and the map induced by the inclusion of $Z$,
	respectively. Since $Z$ is central in $\Spin(q)$, the connecting map $\delta_K$ is a group
	homomorphism; see \cite[Chapter~I, \S5.6, Corollary~2]{SerreGC}. The Galois cohomology sequence associated to \eqref{eq:spin-central-isogeny} gives
	\begin{equation}\label{eq:delta-alpha}
		\delta_K(\PGOp(q)(K))=\ker\alpha_K;
	\end{equation}
	see \cite[Chapter~I, \S5.7, Proposition~43]{SerreGC}.
	
	We denote by $\Omega(q)$ the extended Clifford group associated with $q$; see \cite[Definition~(13.18)]{KMRT1998}.
	
	\section{Preliminaries on the norm principle}\label{sec:norm-prelim}
	
	Let $G$ be a reductive $F$-group, let $T$ be an $F$-torus, and let
	$\varphi\colon G\to T$ be an $F$-homomorphism. For a finite separable
	extension $L/F$, the norm principle for $\varphi$ over $L/F$ is
	the inclusion
	\[
	N_{T,L/F}(\varphi(G(L)))\subset\varphi(G(F)).
	\]
	We say that the norm principle holds for $G$ over $L/F$ if the norm principle holds for every $F$-homomorphism from $G$ to an $F$-torus.
	
	The norm principle for $\Omega(q)$ admits the following equivalent descriptions.	
	
	\begin{lemma}[Bhaskhar--Chernousov--Merkurjev]
		\label{lem:bcm-equivalence}
		Let $q$ be a nondegenerate quadratic form over $F$ of even dimension
		at least $4$, and let $L/F$ be a finite separable extension. The following conditions are equivalent:
		\begin{enumerate}[label=\textup{(\roman*)}]
			\item The norm principle holds for $\Omega(q)$ over $L/F$.
			\item One has $\Cor_{L/F}(\delta_L(\PGOp(q)(L)))\subset\delta_F(\PGOp(q)(F))$.
			\item One has $\Cor_{L/F}(\ker\alpha_L)\subset\ker\alpha_F$.
		\end{enumerate}
	\end{lemma}
	
	\begin{proof}
		See \cite[Lemma~2.2]{BCM2019}.
	\end{proof}
	
	The next theorem reduces Theorem~\ref{thm:main} to the case when $L/F$ is quadratic.
	
	\begin{theorem}[Bhaskhar--Chernousov--Merkurjev]\label{thm:quadratic-reduction}
		Let $q$ be a nondegenerate quadratic form over $F$ of even dimension
		at least $4$. Suppose that, for every finite separable extension $K/F$,
		the norm principle holds for $\Omega(q)_K$ over every quadratic
		extension of $K$. Then the norm principle holds for $\Omega(q)$ over
		every finite separable extension of $F$.
	\end{theorem}
	
	\begin{proof}
		See \cite[Theorem~2.7]{BCM2019}.
	\end{proof}
	
	Let $X$ be an $F$-variety. Recall that two rational points $x_0,x_1\in X(F)$ are directly $R$-equivalent if there is a rational map $\varphi\colon\mathbb P^1_F\dashrightarrow X$, defined at $0$ and $1$, with $\varphi(0)=x_0$ and $\varphi(1)=x_1$. The equivalence relation
	generated by this relation is called $R$-equivalence.
	
	For a smooth affine $F$-group $G$, the points $R$-equivalent to the identity form a normal subgroup $RG(F)$, and $G(F)/R$ denotes the quotient by this subgroup.
	
	\begin{remark}\label{rem:cayley}
		Let $F$ be a field of characteristic not equal to $2$, and let $q$ be a nondegenerate quadratic form over $F$. By the Cayley parametrization, $\SO(q)$ is $F$-rational; see \cite{Cayley1846} or \cite[\S 2.4]{GilleLectures}. In particular, for every field extension $K/F$ such that $K$ is infinite, the group $\SO(q)(K)/R$ is trivial.
	\end{remark}	
	
	\begin{theorem}[Gille]\label{thm:gille}
		Let
		\[
		1\longrightarrow C\longrightarrow\widetilde G\longrightarrow G\longrightarrow1
		\]
		be a central short exact sequence of $F$-groups, where $\widetilde G$ and $G$ are semisimple and $C$ is finite \'etale. For a field extension $K/F$, let
		$\partial_K\colon G(K)\to H^1(K,C)$ be the connecting map. Then,
		for every finite separable extension $L/F$,
		\[
		\Cor_{L/F}(\partial_L(RG(L)))\subset\partial_F(RG(F)).
		\]
	\end{theorem}
	
	\begin{proof}
		See \cite[Theorem~A]{Gille1997} or \cite[Theorem~4.1]{BCM2019}.
	\end{proof}
	
	The following consequence of Theorem~\ref{thm:gille} will be used in the proof of Theorem~\ref{thm:main}.
	
	\begin{corollary}\label{cor:r-torus-corestriction}
		Let
		\[
		1\longrightarrow C\longrightarrow \widetilde G\longrightarrow G\longrightarrow 1
		\]
		be a central short exact sequence of $F$-groups, where $\widetilde G$
		and $G$ are semisimple and $C$ is finite \'etale. For every field
		extension $K/F$, let
		\[
		\partial_K\colon G(K)\longrightarrow H^1(K,C)
		\]
		be the connecting homomorphism.
		
		Let $L/F$ be a finite separable extension, and suppose that $x\in G(L)$ admits a factorization $x=rt$, where $r\in RG(L)$ and $t\in T(L)$ for some $F$-torus $T\subset G$. Then there exists $r_0\in RG(F)$ such that \[\Cor_{L/F}(\partial_L(x))= \partial_F(r_0N_{T,L/F}(t))\quad \text{in $H^1(F,C)$}.\] In particular, $\Cor_{L/F}(\partial_L(x))
		\in \partial_F(G(F))$.
	\end{corollary}
	
	\begin{proof}
		Since $C$ is central, the maps $\partial_K$ are homomorphisms.
		By Theorem~\ref{thm:gille}, there exists $r_0\in RG(F)$ such that $\Cor_{L/F}(\partial_L(r))=\partial_F(r_0)$.
		
		Let $\widetilde T\subset\widetilde G$ be the inverse image of $T$.
		Then
		\[
		1\longrightarrow C\longrightarrow\widetilde T\longrightarrow T\longrightarrow 1
		\]
		is a central exact sequence, and its connecting homomorphism is the
		restriction of $\partial_K$ to $T(K)$. Compatibility of
		corestriction with connecting homomorphisms gives $\Cor_{L/F}(\partial_L(t))=\partial_F(N_{T,L/F}(t))$.
		Therefore
		\begin{align*}
			\Cor_{L/F}(\partial_L(x))&=\Cor_{L/F}(\partial_L(r))+\Cor_{L/F}(\partial_L(t))\\
			&=\partial_F(r_0)+\partial_F(N_{T,L/F}(t))\\
			&=\partial_F(r_0N_{T,L/F}(t)).\qedhere
		\end{align*}
	\end{proof}

	\section{\'Etale algebras with involution and maximal tori}
	\label{sec:etale-tori}
	
	The description of maximal tori in $\SO(q)$ by \'etale algebras with involution is standard; see \cite[Proposition~3.3]{BCM03} or \cite[Proposition~1.2.1]{BF14}. Here we consider maximal tori in $\GOp(q)$.
	
	\begin{lemma}\label{lem:etale-similitude-torus}
		Let $F$ be a field of characteristic different from $2$, let $(V,q)$ be a nondegenerate quadratic space over $F$ of dimension $2n\geq2$, and let
		$E\subset\End_F(V)$ be a finite \'etale $F$-algebra of dimension $2n$, stable under $\operatorname{ad}_q$. Let $\tau=(\operatorname{ad}_q)|_E$, and suppose that $\dim_F E^\tau=n$. Then $V$ is free of rank one over $E$, the functor on commutative $F$-algebras $R$ given by
		\[
		S(R)=\{z\in(E\otimes_F R)^\times:z\tau(z)\in R^\times\}
		\]
		is a maximal $F$-torus of $\GOp(q)$, and the restriction to $S$ of the multiplier map is given by $z\mapsto z\tau(z)$.
	\end{lemma}
	
	\begin{proof}
		Write $E=\prod_i E_i$, where the $E_i$ are finite separable field extensions of $F$. The faithful action of $E$ on $V$ gives a decomposition $V=\bigoplus_i V_i$ with each $V_i$ a nonzero $E_i$-vector space. Therefore
		\[
		2n=\dim_F V=\sum_i[E_i:F]\dim_{E_i}V_i
		\geq\sum_i[E_i:F]=2n.
		\]
		Thus $\dim_{E_i}V_i=1$ for every $i$, proving that $V$ is free of rank one over $E$.
		
		Fix a separable closure $F_s$ of $F$. Since $E$ is \'etale of dimension $2n$, we have $E\otimes_FF_s\simeq F_s^{2n}$. The involution $\tau$ permutes the $2n$ factors. Since $\dim_F E^\tau=n$, no factor of $E\otimes_FF_s$ is fixed by $\tau$, and so after reordering we may use coordinates
		$(x_1,y_1,\ldots,x_n,y_n)$, with $\tau$ exchanging $x_i$ and $y_i$.
		Thus
		\[
		S_{F_s}
		\simeq \{(x_1,y_1,\ldots,x_n,y_n)\in\Gm^{2n}:x_1y_1=\cdots=x_ny_n\}\simeq \Gm^{n+1}.
		\]
		Thus $S$ is an $(n+1)$-dimensional $F$-torus.
		
		Since $\tau$ is the adjoint involution on $E$, an element $z\in S(R)$
		is a similitude of multiplier $\mu=z\tau(z)$. Therefore
		\[
		\det(z)=N_{E/F}(z)
		=N_{E^\tau/F}(z\tau(z))=\mu^n.
		\]
		This shows that $S\subset\GOp(q)$. The group $\GOp(q)$ has rank $n+1$, so $S$ is a  maximal $F$-torus of $\GOp(q)$.
	\end{proof}

	\begin{proposition}\label{prop:spectral-torus}
		Let $F$ be an infinite field of characteristic different from $2$, let $L/F$ be a quadratic extension with nontrivial automorphism $\sigma$, and let $(V,q)$ be a nondegenerate quadratic space over $F$ of dimension $2n\geq2$. Suppose that $P\in\operatorname{GL}(V_L)$ and $\nu\in F^\times$ satisfy $\sigma(P)=P^*=\nu P^{-1}$ and that the characteristic polynomial $\chi_P(t)$ is separable and prime to $t^2-\nu$. Let $L[P]\subset \End_L(V_L)$ be the $L$-subalgebra generated by $P$, and set
		\[
		E\coloneqq L[P]^\sigma\subset\End_F(V).
		\]
		Then the following statements hold.
		\begin{enumerate}
			\item The $F$-algebra $E$ is finite \'etale of dimension $2n$ and stable under $\operatorname{ad}_q$. Letting $\tau=(\operatorname{ad}_q)|_E$, we have $\dim_F E^\tau=n$ and
			$E\otimes_F L=L[P]$.
			\item The $E$-module $V$ is free of rank one. In particular,
			\[
			\operatorname{Cent}_{\End_L(V_L)}(L[P])=L[P].
			\]
			\item There exists $r\in\operatorname O(q)(F)$ such that
			\[
			r^2=1,\qquad rzr^{-1}=\tau(z)\quad \forall z\in E,
			\qquad\det(r)=(-1)^n.
			\]
		\end{enumerate}
		In particular, by Lemma~\ref{lem:etale-similitude-torus}, the functor on $F$-algebras $R$
		\begin{equation}\label{eq:definition-of-s}
		S(R)=\{z\in(E\otimes_F R)^\times:z\tau(z)\in R^\times\}
		\end{equation}
		is a maximal $F$-torus of $\GOp(q)$, with multiplier map $z\mapsto z\tau(z)$.
	\end{proposition}
	
	\begin{proof}
		(1) Since $\chi_P(t)$ is separable of degree $2n$, it is also
		the minimal polynomial of $P$. Thus $L[P]\simeq L[t]/(\chi_P(t))$ is a finite \'etale $L$-algebra of dimension $2n$. The identities $\sigma(P)=P^*=\nu P^{-1}$ show that $L[P]$ is stable under both $\sigma$ and $\operatorname{ad}_q$. By Speiser's lemma \cite[Lemma~2.3.8]{GS17}, the natural map $E\otimes_F L\to L[P]$ is an isomorphism. Hence $E$ is finite \'etale of dimension $2n$. Since $q$ is defined over $F$, its adjoint involution $\operatorname{ad}_q$ commutes with $\sigma$ and therefore restricts to an $F$-algebra involution $\tau$ on $E$.
		
		Fix a separable closure $F_s$ containing $L$, and let $\Sigma\subset F_s$ be the set of eigenvalues of $P$. Then
		\[
		E\otimes_F F_s\simeq L[P]\otimes_L F_s \simeq\prod_{z\in\Sigma}F_s.
		\]
		The involution $\tau$ permutes these factors by $z\mapsto\nu/z$.
		Because $\chi_P$ is prime to $t^2-\nu$, this permutation has no fixed points. Therefore, the $2n$ factors are exchanged in $n$ pairs, and their fixed subalgebra has dimension $n$. Since $\mathrm{char}(F)\neq 2$ and $\sigma$ has order $2$, taking $\sigma$-invariants commutes with scalar extension, so $\dim_F E^\tau=n$.
		
		(2) By Lemma~\ref{lem:etale-similitude-torus}, $V$ is free of rank one over $E$. Thus $V_L$ is therefore free of rank one over
		$E\otimes_F L=L[P]$, and hence
		\[
		\operatorname{Cent}_{\End_L(V_L)}(L[P])
		=\End_{L[P]}(V_L)=L[P],
		\]
		where the last identification sends an element of $L[P]$ to multiplication by that element.
		
		(3) Using (2), choose an $E$-linear isomorphism
		$\iota\colon E\xrightarrow{\sim} V$, and use it define a bilinear form $b_E$ on $E$:
		\[
		b_E(x,y)\coloneqq b(\iota(x),\iota(y))\qquad \forall x,y\in E.
		\]
		The adjoint identity becomes
		\[
		b_E(zx,y)=b_E(x,\tau(z)y)\qquad \forall x,y,z\in E.
		\]
		Define the $F$-linear map $\ell\colon E\to F$ by $\ell(z)=b_E(1,z)$.
		Applying the preceding identity to $b_E(x\cdot1,y)$ gives
		\[
		b_E(x,y)=\ell(\tau(x)y)\qquad \forall x,y\in E.
		\]
		Moreover, for every $z\in E$ symmetry of $b_E$ gives
		\[
		\ell(\tau(z))=b_E(1,\tau(z))=b_E(z,1)=b_E(1,z)=\ell(z).
		\]
		It follows that, for all $x,y\in E$,
		\[
		b_E(\tau(x),\tau(y))
		=\ell(x\tau(y))
		=\ell(\tau(x\tau(y)))=\ell(\tau(x)y)=b_E(x,y).
		\]
		Thus $\tau$, viewed as an $F$-linear automorphism of $E$, is an isometry.
		Set
		\[
		r\coloneqq\iota\circ\tau\circ\iota^{-1}\in\operatorname O(q)(F).
		\]
		Clearly $r^2=1$. For all $x,z\in E$, using the $E$-linearity of $\iota$, we
		have
		\[
		rzr^{-1}(\iota(x))
		=\iota(\tau(z\tau(x)))
		=\iota(\tau(z)x)
		=\tau(z)\iota(x),
		\]
		so $rzr^{-1}=\tau(z)$. Finally, $r$ has the same determinant as $\tau$
		acting on $E$. Its $+1$-eigenspace is $E^\tau$, of dimension $n$, and
		its $-1$-eigenspace also has dimension $n$. Hence
		$\det(r)=(-1)^n$.
	\end{proof}

	\section{Self-adjoint similitudes and generic products}\label{sec:self-adjoint}

	\begin{lemma}\label{lem:self-adjoint-root}
		Let $K$ be a field of characteristic different from $2$, let $(V,q)$
		be a nondegenerate quadratic space of dimension $2n\geq2$, and let
		$\lambda\in K^\times\setminus K^{\times2}$ be a multiplier of a
		similitude of $q$. There exists $A\in\GO(q)(K)$ such that $A^*=A$, $A^2=\lambda I$ and $\chi_A(t)=(t^2-\lambda)^n$.
	\end{lemma}
	
	\begin{proof}
		By \cite[Theorem~1]{Vinroot06}, a similitude of multiplier $\lambda$
		can be written as $\rho A$, where $\rho\in \mathrm{O}(q)(K)$ is an orthogonal involution,
		$A^2=\lambda I$, and $\mu(A)=\lambda$. Since $\rho^*\rho=I$, we have $A^*A=\lambda I=A^2$, and hence multiplication on the right by $A^{-1}$ gives $A^*=A$. 
		
		Let $t$ be a variable over $K$, and consider the quadratic field extension $K' \coloneqq K[t]/(t^2-\lambda)$. Since $A^2=\lambda I$, the $K[t]$-action on $V$ given by letting $t$ act via $A$ makes $V$ into a $K'$-vector space of dimension $n$. Fix a $K'$-basis $v_1,\dots,v_n$ of $V$, so that $V=\oplus_{i=1}^nK'v_i$. For every $i=1,\dots,n$, the $K$-vector space $K'v_i$ has the $K$-basis $v_i,tv_i$. Since $Av_i=tv_i$ and $A(tv_i)=t^2v_i=\lambda v_i$, the characteristic polynomial of the restriction of $A$ to $K'v_i$ is $t^2-\lambda$. It follows that $\chi_A(t)=(t^2-\lambda)^n$.
	\end{proof}

	\begin{proposition}\label{prop:regular-product}
		Let $L/F$ be a quadratic extension with nontrivial automorphism $\sigma$,
		let $(V,q)$ be a quadratic space over $F$ of dimension $2n\geq2$, and
		suppose that $F$ is infinite. Let
		$\lambda\in L^\times\setminus L^{\times2}$ be a multiplier of a
		similitude of $q_L$, and put $\nu\coloneqq \lambda\sigma(\lambda)\in F^\times$. Then there exists $A\in\GO(q)(L)$ such that, setting $P\coloneqq A\sigma(A)$, the following holds.
		\begin{enumerate}
			\item[(i)] We have $A^*=A$, $A^2=\lambda I$, and $\chi_A(t)=(t^2-\lambda)^n$.
			\item[(ii)] We have $\sigma(P)=P^*=\nu P^{-1}$.
			\item[(iii)] The characteristic polynomial $\chi_P(t)\in F[t]$ is separable and prime to $t^2-\nu$.
		\end{enumerate}
	\end{proposition}
	
	\begin{proof}
		By Lemma~\ref{lem:self-adjoint-root}, choose $A_0\in\GO(q)(L)$ such that
		$A_0^*=A_0$, $A_0^2=\lambda I$, and
		$\chi_{A_0}(t)=(t^2-\lambda)^n$. For every $g\in\SO(q)(L)$, set
		\[A_g\coloneqq gA_0g^{-1},\qquad P_g\coloneqq A_g\sigma(A_g).\]
		Conjugation by an isometry preserves the adjoint involution, so $A_g^*=A_g$, $A_g^2=\lambda I$, and $\chi_{A_g}(t)=(t^2-\lambda)^n$. Moreover,
		\[
		P_g^*=\sigma(A_g)A_g=\sigma(P_g)=\nu P_g^{-1}.
		\]
		Indeed, the last equality follows from
		$P_g\sigma(P_g)=A_g\sigma(A_g)^2A_g=\nu I$. Since an operator and
		its adjoint have the same characteristic polynomial, we also have
		$\sigma(\chi_{P_g})=\chi_{P_g}$. Thus $A_g$ and $P_g$ satisfy (i) and (ii) for every $g\in\SO(q)(L)$.
		
		We now show that $g$ may be chosen so that $\chi_{P_g}(t)$ is separable
		and prime to $t^2-\nu$, that is, so that (iii) is also satisfied. Consider the Weil restriction
		$W\coloneqq R_{L/F}(\SO(q)_L)$, and let $B\coloneqq \mathcal{O}(W)$, so that $W=\mathrm{Spec}(B)$.
		Let
		\[g_{\mathrm{univ}}\in W(B)=\SO(q)(L\otimes_FB)\] 
		be the universal point of $W$, corresponding to $\operatorname{id}_W\colon W\to W$. The coefficients of
		$\chi_{P_{g_{\mathrm{univ}}}}$ lie in
		$L\otimes_FB$, and so the identity
		$\sigma(\chi_{P_{g_{\mathrm{univ}}}})
		=\chi_{P_{g_{\mathrm{univ}}}}$ shows that they belong to
		$(L\otimes_FB)^\sigma=B$. It follows that \[D(g)\coloneqq\operatorname{Disc}_t(\chi_{P_g}(t)),\qquad Q(g)\coloneqq\operatorname{Res}_t(\chi_{P_g}(t),t^2-\nu)\] belong to $B$. Let $U\subset W$ be the principal open subscheme defined by $D(g)Q(g)\neq0$. Thus $U(F)$ consists precisely of those $g\in\SO(q)(L)$ for which $\chi_{P_g}(t)$ is separable and prime to $t^2-\nu$. 
		
		We show that $U$ is nonempty. Fix an algebraic closure $\overline F$ containing $L$. The two embeddings
		$L\hookrightarrow\overline F$ induce an $\overline F$-isomorphism
		\[
		W_{\overline F}\simeq
		\SO(q)_{\overline F}\times\SO(q)_{\overline F}
		\]
		under which the conjugation action on the pair
		$(A_0,\sigma(A_0))$ becomes
		\[
		(g_1,g_2)\cdot(A_0,\sigma(A_0))
		=(g_1A_0g_1^{-1},
		g_2\sigma(A_0)g_2^{-1}).
		\]
		It is therefore enough to find operators $A_1$ and $A_2$, conjugate under $\SO(q)(\overline F)$ to $A_0$ and $\sigma(A_0)$, respectively, such that $A_1A_2$ has distinct eigenvalues, and such that no eigenvalue of $A_1A_2$ has square $\nu$.
		
		Choose $\alpha,\beta\in\overline F^\times$ with
		$\alpha^2=\lambda$ and $\beta^2=\sigma(\lambda)$, and decompose
		$V_{\overline F}$ into $n$ orthogonal hyperbolic planes with Gram
		matrix $H=\left(\begin{smallmatrix}0&1\\1&0\end{smallmatrix}\right)$.
		Choose $t_1,\ldots,t_n\in\overline F^\times$ such that the $2n$
		elements $t_i,t_i^{-1}$ are pairwise distinct and none is equal to
		$\pm1$. On the $i$th hyperbolic plane define
		\[
		A_1=\alpha
		\begin{pmatrix}0&1\\1&0\end{pmatrix},
		\qquad
		A_2=\beta
		\begin{pmatrix}0&t_i\\t_i^{-1}&0\end{pmatrix}.
		\]
		Both operators are self-adjoint for the adjoint involution determined by $H$. Moreover, $A_1^2=\lambda I$ and $A_2^2=\sigma(\lambda)I$, and the eigenvalues $\pm\alpha$ of $A_1$ and $\pm\beta$ of $A_2$ each occur with multiplicity $n$.
		
		We claim that $A_1$ is $\SO(q)(\overline F)$-conjugate to $A_0$. Indeed, the corresponding eigenspaces of $A_0$ and $A_1$ are nondegenerate of the same dimension, so $A_0$ and $A_1$ are $\operatorname O(q)(\overline F)$-conjugate. If an isometry conjugating $A_0$ to $A_1$ has determinant $-1$, composing it with a reflection in an eigenspace of $A_0$ gives an isometry conjugating $A_0$ to $A_1$ with determinant $1$. Thus the isometry which conjugates $A_0$ to $A_1$ may be chosen in $\SO(q)(\overline F)$. Similarly, $A_2$ is $\SO(q)(\overline F)$-conjugate to $\sigma(A_0)$.
		
		On the $i$th hyperbolic plane one has
		\[
		A_1A_2=\alpha\beta
		\begin{pmatrix}t_i^{-1}&0\\0&t_i\end{pmatrix}.
		\]
		Hence the eigenvalues of $A_1A_2$ are the $2n$ pairwise distinct
		elements $\alpha\beta t_i^{\pm1}$, and moreover none of them has square $\nu$ because  $t_i\neq\pm1$. Thus $U(\overline F)$, and hence
		$U$, is nonempty.
		
		Finally, by Remark~\ref{rem:cayley},
		$\SO(q)_L$ is $L$-rational. It follows that the Weil restriction $W$ is $F$-rational. As $F$ is infinite, this implies that $U(F)\neq\emptyset$. Choose
		$g\in U(F)$ and set $A\coloneqq A_g$ and
		$P\coloneqq P_g$. Then $A$ and $P$ satisfy properties (i)-(iii).
	\end{proof}
	
	\section{Proofs of Theorems~\ref{thm:main} and \ref{thm:toral-replacement}}
	\label{sec:toral-replacement}
	
	\begin{proof}[Proof of Theorem~\ref{thm:toral-replacement}]
		Let $f\in\GOp(q)(L)$, and set $\lambda=\mu(f)$. If
		$\lambda=c^2$ with $c\in L^\times$, we may take $s=cI$. Suppose now that $\lambda$ is nonsquare, and let $A\in \GO(q)(L)$ and $P\coloneqq A\sigma(A)$ be as in Proposition~\ref{prop:regular-product}. Set $E\coloneqq L[P]^\sigma$, and let $\tau$, $r$, and $S$ be as in Proposition~\ref{prop:spectral-torus}.
		
		Conjugation by $A$ induces the adjoint involution on $L[P]$: indeed, 
		using $A^2=\lambda I$, we have
		\[
		APA^{-1}=A^2\sigma(A)A^{-1}=\sigma(A)A=P^*=\mathrm{ad}_q(P).
		\]
		Since conjugation by $A$ and the adjoint involution are $L$-algebra
		automorphisms of the commutative algebra $L[P]$, it follows that they coincide on $L[P]$.
		Conjugation by $r$ also induces $\mathrm{ad}_q$ on $L[P]$.
		It follows that $s\coloneqq Ar$ centralizes $L[P]$, and hence by Proposition~\ref{prop:spectral-torus}(2) that $s\in L[P]$. Moreover, $s^{-1}$ also
		centralizes $L[P]$, so $s\in L[P]^\times=(E\otimes_F L)^\times$.
		Since $r$ is an isometry and $\mu(A)=\lambda$, we have $\tau(s)s=s^*s=\lambda I$. Therefore $\mu(s)=\lambda$ and, by \eqref{eq:definition-of-s}, $s\in S(L)$. In particular, $s\in \GOp(q)(L)$. Thus $u\coloneqq fs^{-1}$ is a proper similitude of multiplier $1$, that is, $u\in\SO(q)(L)$, proving the theorem.
	\end{proof}
	
	\begin{proof}[Proof of Theorem~\ref{thm:main}]
		If $F$ is finite, by Lang's theorem \cite{Lang1956} the pointed set $H^1(F,\Spin(q))$ is trivial. Hence $\ker\alpha_F=H^1(F,Z)$, and the conclusion follows from
		Lemma~\ref{lem:bcm-equivalence}.
		
		Assume now that $F$ is infinite. By
		Theorem~\ref{thm:quadratic-reduction}, it is enough to consider a
		quadratic extension $L/F$. Let $h\in\PGOp(q)(L)$. The exact sequence
		\[
		1\longrightarrow\Gm\longrightarrow\GOp(q)
		\longrightarrow\PGOp(q)\longrightarrow1
		\]
		and Hilbert's theorem~90 give a lift
		$f\in\GOp(q)(L)$ of $h$.
		
		By Theorem~\ref{thm:toral-replacement}, we have $f=us$, where $u\in\SO(q)(L)$ and $s\in S(L)$ for some maximal $F$-torus
		$S\subset\GOp(q)$. Let $T\subset\PGOp(q)$ be the image of $S$, and let $\overline u\in \PGOp(q)(L)$ and $\overline s\in T(L)$ denote the image of $u$ and $s$, respectively, so that $h=\overline u\,\overline s$.
		
		Since $L$ is infinite, Remark~\ref{rem:cayley} gives
		$\SO(q)(L)/R=1$. As morphisms of algebraic groups preserve $R$-equivalence, we have $\overline u\in R\PGOp(q)(L)$. Applying Corollary~\ref{cor:r-torus-corestriction} to the central isogeny \eqref{eq:spin-central-isogeny} and to the factorization $h=\overline u\,\overline s$, we obtain
		\[
		\Cor_{L/F}(\delta_L(h))
		\in \delta_F(\PGOp(q)(F)).
		\]
		Thus condition \textup{(ii)} of
		Lemma~\ref{lem:bcm-equivalence} holds. Hence the norm principle
		holds for $\Omega(q)$ over $L/F$, as desired.
	\end{proof}
	
	We conclude with two consequences of Theorem~\ref{thm:main} for more general reductive
	$F$-groups. Let $G$ be a reductive $F$-group, let $G^{\mathrm{sc}}$ be the simply connected cover of the derived subgroup of $G$, and write
	\begin{equation}\label{eq:g-decomposition}
		G^{\mathrm{sc}}\simeq
		\prod_{i=1}^r R_{F_i/F}(G_i),
	\end{equation}
	where $F_i/F$ are finite separable extensions and the $G_i$ are absolutely simple simply connected $F_i$-groups.
	
	\begin{corollary}\label{cor:reductive-norm-principle}
		Let $F$ be a field of characteristic different from $2$, and let $G$ be
		a reductive $F$-group. Assume that, in the decomposition
		\eqref{eq:g-decomposition}, no $G_i$ is of type $E_6$ or $E_7$, and
		that, for every $G_i$ of type $D_n$ for some $n\geq4$, there exists a
		nondegenerate quadratic form $q_i$ over $F_i$ of dimension $2n$ such
		that $G_i\simeq\Spin(q_i)$. Then the norm principle holds for $G$ over
		every finite separable extension of $F$.
	\end{corollary}
	
	\begin{proof}
		For the factors of type $D_n$, the group $\Omega(q_i)$ is an envelope of
		$\Spin(q_i)$, and the norm principle for $\Omega(q_i)$ follows from
		Theorem~\ref{thm:main}. For all the remaining factors, it follows from \cite[Theorem~1.1]{BM00}. Now apply \cite[Proposition~5.2]{BM00}.
	\end{proof}
	
	Thus Theorem~\ref{thm:main} extends the theorem of
	Barquero--Merkurjev \cite[Theorem~1.1]{BM00} by allowing type $D_n$
	factors arising from arbitrary nondegenerate quadratic forms.
	
	We also obtain new cases of Serre's injectivity question \cite[p.~233, Question~2]{Serre1995}; see also \cite[Question~1.1]{Bhaskhar2016}.
	
	\begin{corollary}\label{cor:serre-injectivity}
		Let $F$ be a field of characteristic different from $2$, and let $G$ be
		a connected reductive $F$-group. Assume, in the decomposition
		\eqref{eq:g-decomposition}, that every $G_i$ is of classical type
		$A_n$, $B_n$, $C_n$, or $D_n$, and that every factor of type $D_n$ is
		isomorphic to $\Spin(q_i)$ for a nondegenerate quadratic form $q_i$
		over $F_i$. Given finite extensions $L_1/F,\dots,L_m/F$ of collectively coprime
		degrees, the kernel of
		\[
		H^1(F,G)\longrightarrow \prod_{j=1}^m H^1(L_j,G)
		\]
		is trivial. In particular, every $G$-torsor over $F$ admitting a
		zero-cycle of degree $1$ has an $F$-rational point.
	\end{corollary}
	
	\begin{proof}
		Apply the reduction of \cite[\S\S2--3]{Bhaskhar2016}. In the proof of
		\cite[Theorem~1.2]{Bhaskhar2016}, replace the norm
		principle of Barquero--Merkurjev by Corollary~\ref{cor:reductive-norm-principle}. This gives a
		positive answer to Serre's injectivity question for $G$, and the zero-cycle
		formulation follows as in \cite[\S2.1]{Bhaskhar2016}.
	\end{proof}

\end{document}